\documentclass{amsart}
\usepackage{amssymb}
\usepackage{amsmath}

\newcommand{\map}[1]{\xrightarrow{#1}}

\newcommand{\iso}{\cong}

\newcommand{\Q}{\mathbb Q}
\newcommand{\Z}{\mathbb Z}
\newcommand{\C}{\mathbb C}
\newcommand{\alg}{\mathrm{alg}}
\newcommand{\A}{\mathbb A}
\newcommand{\co}{\mathcal O}

\begin{document}

\title{Central derivatives of L-functions in Hida families}

\author{Benjamin Howard}
\address{Department of Mathematics, Boston College, Chestnut Hill, MA, 02467}
\email{howardbe@bc.edu}

\begin{abstract}
We prove a result of the following type: given a Hida family of modular forms, if there exists a weight two form in the family whose $L$-function vanishes to exact order one at $s=1$,  then all but finitely many weight two forms in the family enjoy this same property.  The analogous result for order of vanishing zero is also true, and is an easy consequence of the existence of the Mazur-Kitagawa two-variable $p$-adic $L$-function.  
\end{abstract}

\thanks{This research was supported in part  by NSF grant DMS-0556174}

\maketitle

\newtheorem{Thm}{Theorem}
\newtheorem{Prop}[Thm]{Proposition}
\newtheorem{Lem}[Thm]{Lemma}
\newtheorem{Cor}[Thm]{Corollary}
\newtheorem{Def}[Thm]{Definition}

\renewcommand{\labelenumi}{(\alph{enumi})}

%%%%%%%%%%%%%%%%%%%%%%%%%%%%%%%%%%%%%%%%%%%%%%

\section{Introduction}
\label{sec:0}

%%%%%%%%%%%%%%%%%%%%%%%%%%%%%%%%%%%%%%%%%%%%%%%%

Fix embeddings $\Q^\alg\hookrightarrow \Q_p^\alg$ and $\Q^\alg\hookrightarrow \C$, and let $\omega:(\Z/p\Z)^\times\rightarrow \mu_{p-1}$ denote the Teichmuller lift.  Suppose that $M$ is a positive integer, $p\nmid M$ is  prime, and 
\begin{equation}\label{cusp form}
g=\sum_{n>0} a(n) q^n \in S_k(\Gamma_0(Mp),\omega^j,\Q^\alg)
\end{equation}
is a normalized cusp form of weight $k\ge 2$, level $\Gamma_0(Mp)$, and nebentype $\omega^j$.  We assume that $g$ is an eigenform  for all Hecke operators $T_\ell$ with $\ell\nmid Mp$ and $U_\ell$ with $\ell\mid Mp$, that $g$ is $M$-new, and that $g$ is $p$-ordinary in the sense that $|a(p)|_p=1$.   Fixing a finite extension $F/\Q_p$ inside of $\Q_p^\alg$ large enough to contain all Fourier coefficients of $g$, let $\rho_g:G_\Q\rightarrow \mathrm{GL}_2(F)$ be the Galois representation attached to $g$ by Deligne.  We assume throughout  that $p$ does not divide $6M\phi(M)$ (where $\phi$ is Euler's function) and that the residual Galois representation of $\rho_g$ is absolutely irreducible. 

To the form $g$ Hida attaches a $q$-expansion $\mathbf{g}\in R[[q]]$ where $R$ is an integral domain, finite and flat over the Iwasawa algebra $\Lambda=\co_F[[1+p\Z_p]]$ ($\co_F$ is the maximal order in $F$).  If $\mathfrak{p}$ is an arithmetic prime of $R$ as defined in \S \ref{sec:1} and $F_\mathfrak{p}=R_\mathfrak{p}/\mathfrak{p}$, then the image of $\mathbf{g}$ in $F_\mathfrak{p}[[q]]$, denoted $g_\mathfrak{p}$,  is the $q$-expansion of a $p$-adic modular form of some weight, level, and nebentype all of which depend on $\mathfrak{p}$.  For a suitable choice of $\mathfrak{p}$ we recover $g=g_\mathfrak{p}$.   Let us denote by $\mathfrak{P}_2$ the set of all weight two arithmetic primes of $R$; that is, the (infinite) set of arithmetic primes $\mathfrak{p}$ for which the $p$-adic modular form $g_\mathfrak{p}$ has weight two.   The family $\{g_\mathfrak{p}\mid \mathfrak{p}\in\mathfrak{P}_2\}$ contains  only finitely many forms of any given nebentype, but there is a uniform way of twisting the entire family  into a family of $p$-adic modular forms of trivial nebentype.  Indeed, a choice of critical character $\theta:\Z_p^\times\rightarrow R^\times$ in the sense of \S \ref{sec:1} provides us with, for each $\mathfrak{p}\in\mathfrak{P}_2$,  a character $\vartheta_\mathfrak{p}: \Z_p^\times\rightarrow F_\mathfrak{p}^\times$ whose square is the nebentype of $g_\mathfrak{p}$.  Thus the form $f_\mathfrak{p}=g_\mathfrak{p}\otimes\vartheta_\mathfrak{p}^{-1}$ has trivial nebentype, and one may ask how the order of vanishing (i.e. the \emph{analytic rank}) $\mathrm{ord}_{s=1} L(s,f_\mathfrak{p})$ varies as $\mathfrak{p}$ varies (strictly speaking one must choose an embedding $F_\mathfrak{p}\hookrightarrow \Q_p^\alg$ in order to view the $q$-expansion coefficients of  $f_\mathfrak{p}$ as elements of $\Q^\alg\subset\C$, but we ignore this for the moment).  A conjecture of Greenberg \cite{greenberg_deformation} predicts that as $\mathfrak{p}$ varies over $\mathfrak{P}_2$ all but finitely many of the modular forms $f_\mathfrak{p}$ should have analytic rank zero or one, depending on the root number of the $L$-function of $f_\mathfrak{p}$.  This root number is constant as $\mathfrak{p}$ varies, except for a finite and explicit set of bad $\mathfrak{p}$'s which were determined by Mazur-Tate-Teitelbaum \cite{mtt}.  When the root number is $+1$  it follows easily from the existence of the Mazur-Kitagawa two-variable $p$-adic $L$-function that if one can find a single $\mathfrak{p}\in\mathfrak{P}_2$ for which $f_\mathfrak{p}$ has  analytic rank zero, then  $f_\mathfrak{p}$ has analytic rank zero for all but finitely many $\mathfrak{p}\in\mathfrak{P}_2$; see Theorem \ref{rank zero}, which also includes information about $f_\mathfrak{p}$ of higher even weight.  In the present article we prove a similar result in the case where the root number is $-1$.  If  $ \mathrm{Gen}_2(\theta)\subset\mathfrak{P}_2$  denotes the set of arithmetic primes of weight two which are generic for $\theta$ in the sense of Definition \ref{Def:generic}, then we prove that if there is a single $\mathfrak{p}\in \mathrm{Gen}_2(\theta)$ for which $f_\mathfrak{p}$ has analytic rank one then the analytic rank is one for all but finitely many $\mathfrak{p}\in \mathrm{Gen}_2(\theta)$.  The set $ \mathrm{Gen}_2(\theta)$ is very explicit and contains all but finitely many $\mathfrak{p}\in\mathfrak{P}_2$.  The result is stated in the text as Theorem \ref{central derivatives}.  This provides a method for verifying Greenberg's conjecture in any given case, as one need only  exhibit a single form in the Hida family having the correct analytic rank.

The method of proof is based on the author's earlier papers \cite{howard-hida} and \cite{howard-GZ}.  The first of these papers contains the construction of a \emph{big Heegner point} $\mathfrak{Z}\in H^1(K,\mathbf{T}^\dagger)$ where $K$ is a quadratic imaginary field and $\mathbf{T}^\dagger$ is a twist of the $R$-adic Galois representation attached to $\mathbf{g}$.  The second paper contains an extension of the Gross-Zagier theorem to modular forms of nontrivial nebentype which is general enough to include the forms $\{g_\mathfrak{p} \mid \mathfrak{p}\in\mathrm{Gen}_2(\theta)\}$.  For any $\mathfrak{p}\in \mathrm{Gen}_2(\theta)$ we make explicit the connection between the images 
$
\mathfrak{Z}_\mathfrak{p}\in H^1(K,\mathbf{T}^\dagger\otimes_R F_\mathfrak{p})
$ 
of $\mathfrak{Z}$ for various $\mathfrak{p}$ and the generalized Heegner points constructed in \cite{howard-GZ}.   Assuming that there is a $\mathfrak{p}\in \mathrm{Gen}_2(\theta)$ for which  $f_\mathfrak{p}$ has analytic rank one, we use the main result of \cite{howard-GZ} to prove the nontriviality of $\mathfrak{Z}_\mathfrak{p}$ and then deduce that $\mathfrak{Z}$ is not $R$-torsion (as predicted by \cite[Conjecture 3.4.1]{howard-hida}).  From this it follows that $\mathfrak{Z}_\mathfrak{p}$ is non-trivial for all but finitely many $\mathfrak{p}\in \mathrm{Gen}_2(\theta)$, and again applying the main result of \cite{howard-GZ} we deduce that the  generic analytic rank of the $f_\mathfrak{p}$'s is one.

The sole reason for restricting attention to those forms $f_\mathfrak{p}$ having weight two is that the extension of the Gross-Zagier theorem proved in \cite{howard-GZ} applies only to modular forms of weight two.  While higher even weight analogues of the Gross-Zagier theorem are known \cite{nek:gz,zhang97}, they are not strong enough to deduce higher weight analogues of Proposition \ref{Prop:gz I} below.

If $E$ is a number field we denote by $\A_E$ its ring of adeles and by $$\mathrm{art}_E:\A_E^\times\rightarrow \mathrm{Gal}(E^\mathrm{ab}/E)$$ the Artin reciprocity map $\mathrm{art}_E(x)=[x,E]$ normalized as in \cite[\S5.2]{shimura}.

%%%%%%%%%%%%%%%%%%%%%%%%%%%%%%%%%%%%%%%%%%%%%%

\section{The Hida family}
\label{sec:1}

%%%%%%%%%%%%%%%%%%%%%%%%%%%%%%%%%%%%%%%%%%%%%%%%

In this section we quickly recall the basic facts of Hida theory which we will need, following the notation of \cite[\S 2.1]{howard-hida} (in which the reader will find references to  facts stated here without proof).  
Set $\Gamma=1+p\Z_p$, define 
$$
\Lambda=\co_F[[\Gamma]]\hspace{1cm} \underline{\Lambda}=\co_F[[\Z_p^\times]],
$$
and denote by $z\mapsto [z]$ the canonical inclusion of group-like elements $\Z_p^\times\hookrightarrow\underline{\Lambda}^\times$.  For every integer $t>0$ set 
$$\Phi_t=\Gamma_0(M)\cap\Gamma_1(p^t)\subset\mathrm{SL}_2(\Z)$$ and let $\mathfrak{h}_{r,t}$ denote the $\co_F$-algebra generated by  the Hecke operators 
$$
\{ T_\ell : (\ell,Mp)=1 \} \cup \{ U_\ell : \ell\mid Mp\} \cup \{\langle d\rangle : (d,p)=1\}
$$ 
acting on the space of $p$-adic cusp forms $S_r(\Phi_t,\Q_p^\alg)$ of level $\Phi_t$ and weight $r\ge 2$.  We view each  $\mathfrak{h}_{r,t}$ as a $\underline{\Lambda}$-algebra via  $[z]\mapsto z^{r-2}\langle z\rangle$.  Set $\mathfrak{h}_{r,t}^\mathrm{ord}=e^\mathrm{ord}\mathfrak{h}_{r,t}$ where $e^\mathrm{ord}\in \mathfrak{h}_{r,t}$ is Hida's ordinary projector and define 
$$
\mathfrak{h}^\mathrm{ord}=\varprojlim_t \mathfrak{h}^\mathrm{ord}_{r,t}.
$$  
By Hida's theory the $\underline{\Lambda}$-algebra $\mathfrak{h}^\mathrm{ord}$ is independent of $r$ and is finite and flat over $\Lambda$.

The cuspform (\ref{cusp form}) determines a homomorphism (again denoted $g$) $$\mathfrak{h}^\mathrm{ord}\rightarrow \mathfrak{h}^\mathrm{ord}_{k,1}\rightarrow \co_F.$$  The ring $\mathfrak{h}^\mathrm{ord}$ is semi-local and there is a unique maximal ideal $\mathfrak{m}\subset\mathfrak{h}^\mathrm{ord}$ such that $g$ factors through the local summand $\mathfrak{h}^\mathrm{ord}_\mathfrak{m}$.  The local ring $\mathfrak{h}^\mathrm{ord}_\mathfrak{m}$ then has a unique minimal prime ideal $\mathfrak{a}$ such that $g$ factors through the integral domain 
$
R= \mathfrak{h}^\mathrm{ord}_\mathfrak{m}/\mathfrak{a}.
$
We say that $\mathfrak{h}_\mathfrak{m}^\mathrm{ord}$ is the \emph{Hida family} of $g$, while $R$ is the \emph{branch} of  $\mathfrak{h}_\mathfrak{m}^\mathrm{ord}$ on which $g$ lives.  A prime ideal  $\mathfrak{p}\subset R$ is  \emph{arithmetic} if there is a homomorphism of $\co_F$-algebras $R\rightarrow \Q_p^\alg$ with kernel $\mathfrak{p}$ such that the composition $\Gamma \rightarrow R^\times\rightarrow \Q_p^{\alg,\times}$ has the form $\gamma\mapsto \psi(\gamma)\gamma^{r-2}$ with $\psi$ of finite order and $r\ge 2$ an integer (called the \emph{weight} of $\mathfrak{p}$).  For such a $\mathfrak{p}$ we define $F_\mathfrak{p}=R_\mathfrak{p}/\mathfrak{p}$, a finite extension of $F$.  The homomorphism $\Gamma\rightarrow R^\times\rightarrow F^\times_\mathfrak{p}$ then has the form $\gamma\mapsto \psi_\mathfrak{p}(\gamma) \gamma^{r-2}$ where $\psi_\mathfrak{p}$ is  the \emph{wild character} of $\mathfrak{p}$.        By Hida's theory there is a  formal $q$-expansion
$$
\mathbf{g}=\sum_{n>0} \mathbf{a}(n) q^n\in R[[q]]
$$
with the property that for any arithmetic prime $\mathfrak{p}$ of weight $r$ the image of $\mathbf{g}$ in $F_\mathfrak{p}[[q]]$ is an $M$-new ordinary  $p$-adic eigenform
$$
g_{\mathfrak{p}} = \sum_{n>0} a_\mathfrak{p}(n)q^n \in  S_r(\Gamma_0(Mp^t),  \omega^{k+j-r}\psi_\mathfrak{p}, F_\mathfrak{p} ).
$$
Here $t=t(\mathfrak{p})$ is the smallest positive integer for which $1+p^t\Z_p  \subset  \mathrm{ker}(\psi_\mathfrak{p})$ and we view both $\omega$ and $\psi_\mathfrak{p}$ as characters of 
$$(\Z/p^t\Z)^\times\iso (\Z/p\Z)^\times   \times  (1+p\Z_p)/(1+p^t\Z_p).$$

 A \emph{critical character} is a homomorphism 
 $$ 
 \theta:\Z_p^\times\rightarrow R^\times
 $$
 which satisfies $\theta^2(z)=[z ]$ for all $z\in\Z_p^\times$.  There are exactly two critical characters, and they differ by multiplication by $\omega^{(p-1)/2}$; see Proposition 2.1.3 and Remark 2.1.4 of \cite{howard-hida}.  We now fix, for the remainder of this article, a critical character $\theta$ and define 
$$
\Theta: G_\Q \rightarrow R^\times
$$
by $\Theta(\sigma) =\theta(\epsilon_\mathrm{cyc}(\sigma))$, where $\epsilon_\mathrm{cyc}:G_\Q\rightarrow \Z_p^\times$ is the cyclotomic character.   The homomorphism $[\cdot]:\Z_p^\times\rightarrow R^\times$ has the property that for every arithmetic prime $\mathfrak{p}$ of weight $r$  the composition of $[\cdot]$ with $R^\times\rightarrow F^\times_\mathfrak{p}$ is given by
$$
[z]_\mathfrak{p}=\omega^{k+j-r}(z) \psi_\mathfrak{p}(z) z^{r-2}
$$
for every $z\in\Z_p^\times$.  We denote by $\theta_\mathfrak{p}$ and $\Theta_\mathfrak{p}$ the compositions
$$
\theta_\mathfrak{p}:\Z_p^\times\map{\theta} R^\times\rightarrow  F_\mathfrak{p}^\times
\hspace{1cm}
\Theta_\mathfrak{p}:G_\Q\map{\Theta} R^\times\rightarrow  F_\mathfrak{p}^\times.
$$
and note that when $\mathfrak{p}$ has weight two the nebentype of $g_\mathfrak{p}$ is $\theta_\mathfrak{p}^2=[\cdot]_\mathfrak{p}$,

For each $t>0$ let $X_t=X(\Phi_t)$ be the modular curve over $\mathrm{Spec}(\Q)$ classifying elliptic curves with $\Phi_t$ level structure, set $J_t=\mathrm{Jac}(X_t)$, and define an $\mathfrak{h}^\mathrm{ord}_{2,t}[[G_\Q]]$-module 
$$
\mathrm{Ta}_p^\mathrm{ord}(J_t) = e^\mathrm{ord}(\mathrm{Ta}_p(J_t)\otimes_{\Z_p}\co_F).
$$
Passing to the limit over $t$ we define an $\mathfrak{h}^\mathrm{ord}$-module $\mathbf{Ta}^\mathrm{ord}=\varprojlim \mathrm{Ta}_p^\mathrm{ord}(J_t)$.   According to \cite[Th\'eor\`eme 7]{MT} our hypothesis on the  irreducibility of the residual Galois representation attached to $g$ implies that 
$$
\mathbf{Ta}^\mathrm{ord}_\mathfrak{m} =\mathbf{Ta}^\mathrm{ord}\otimes_{\mathfrak{h}^\mathrm{ord}}\mathfrak{h}^\mathrm{ord}_\mathfrak{m}
\hspace{1cm}
\mathbf{T}=\mathbf{T}^\mathrm{ord}\otimes_{\mathfrak{h}^\mathrm{ord}} R
$$
 are free of rank two over $\mathfrak{h}^\mathrm{ord}_\mathfrak{m}$ and $R$, respectively.  
The $R$-module $\mathbf{T}$ carries a natural $G_\Q$-action, and the twist $\mathbf{T}^\dagger=\mathbf{T}\otimes \Theta^{-1}$ admits a perfect alternating pairing 
$
\mathbf{T}^\dagger\times\mathbf{T}^\dagger\rightarrow R(1)
$
where $R(1)$ is the Tate twist of $R$.   For $E$ any number field denote by 
$$
\mathrm{Sel}_\mathrm{Gr}(E,\mathbf{T}^\dagger)\subset H^1(E,\mathbf{T}^\dagger)
$$
 the \emph{strict Greenberg Selmer group} defined in \cite[Definition 2.4.2]{howard-hida}.  For any arithmetic prime $\mathfrak{p}$ define $V_\mathfrak{p}=\mathbf{T}\otimes_R F_\mathfrak{p}$ and $V^\dagger_\mathfrak{p}=\mathbf{T}^\dagger \otimes_R F_\mathfrak{p}$.

 \begin{Def}
Suppose $f\in S_2(\Gamma_0(M), \Q_p^\alg)$ is a $p$-ordinary newform with $q$-expansion $f=\sum_{n>0} b(n)q^n$ and let $\alpha$ denote the unit root of the polymomial $X^2-b(p)X+p$. The \emph{$p$-stabilization} of $f$ is the $p$-adic modular form of level $Mp$ with $q$-expansion
$$
\sum_{n>0} b(n)q^n+\frac{p}{\alpha} \sum_{n>0}b(n)q^{pn}.
$$
We note that the $p$-stabilization of $f$ is an eigenform for $U_p$ with eigenvalue $\alpha$.
 \end{Def}

\begin{Def}\label{Def:generic}
An arithmetic prime $\mathfrak{p}\subset R$ of weight two is said to be \emph{generic} for  $\theta$ if one of the following (mutually exclusive) hypotheses is satisfied:
\begin{enumerate}
 \item
 $g_\mathfrak{p}$ is the $p$-stabilization of a newform $f\in S_2(\Gamma_0(M),F_\mathfrak{p})$ and $\theta_\mathfrak{p}$ is trivial;
\item
$g_\mathfrak{p}$  is new of level $Mp$ and  $\theta_{\mathfrak{p}}=\omega^{\frac{p-1}{2}}$;
\item
$g_\mathfrak{p}$ has nontrivial nebentype.
 \end{enumerate}
\end{Def}

 Let $\mathrm{Gen}_2(\theta)$ denote the set of all arithmetic weight two primes of $R$ which are generic for $\theta$.  Clearly all but finitely many weight two arithmetic primes are generic for $\theta$ as all but finitely many satisfy condition  (c).

%%%%%%%%%%%%%%%%%%%%%%%%%%%%%%%%%%%%%%%%%%%%

\section{Results over quadratic imaginary fields}
\label{sec:2}

%%%%%%%%%%%%%%%%%%%%%%%%%%%%%%%%%%%%%%%%%%%%

Let $K$ be a quadratic imaginary field of discriminant $-D$ prime to $Mp$ and let $\epsilon_K=\left(\frac{-D}{\cdot}\right)$ denote the quadratic Dirichlet character attached to $K$.  Assume that every prime divisor of $M$ splits in $K$, so that $\epsilon_K(M)=1$, and fix an ideal $\mathfrak{M}\subset \co_K$ of the maximal order of $K$ such that $\co_K/\mathfrak{M}\iso \Z/M\Z$.  Let $H$ denote the Hilbert class field of $K$.  Taking $c=1$ in the construction of \cite[\S 2.2]{howard-hida} gives a canonical cohomology class, the \emph{big Heegner point} of conductor $1$
$$
\mathfrak{X}=\mathfrak{X}_1\in H^1(H,\mathbf{T}^\dagger).
$$
We define
$$
\mathfrak{Z}= \mathrm{Cores}_{H/K}( \mathfrak{X}) \in  H^1(K, \mathbf{T}^\dagger)
$$
and for any arithmetic prime $\mathfrak{p}\subset R$ let $\mathfrak{Z}_\mathfrak{p}$ denote the image of $\mathfrak{Z}$ in $H^1(K,V_\mathfrak{p}^\dagger)$.

For  $\mathfrak{p}\subset R$ arithmetic of weight two, define an $F_\mathfrak{p}^{\times}$-valued character of $\A_K^\times$
$$
\chi_\mathfrak{p} (x)= \Theta_\mathfrak{p}( \mathrm{art}_\Q(\mathrm{N}_{K/\Q}(x)) )
$$
and denote by $\chi_{0,\mathfrak{p}}$ the restriction of $\chi_\mathfrak{p}$ to $\A^\times_\Q$. As before let $t=t(\mathfrak{p})$ be the smallest positive integer for which $\psi_\mathfrak{p}$ is trivial on $1+p^t\Z_p$. The conductor of $\chi_{0,\mathfrak{p}}$ divides $p^t$ (more precisely,  $\chi_{0,\mathfrak{p}}$ is either trivial or of conductor $p^t$), and we may view $\chi_{0,\mathfrak{p}}$ as a character of $(\Z/p^t\Z)^\times$ in the usual way.  The nebentype of $g_\mathfrak{p}$ is  then equal to $\chi_{0,\mathfrak{p}}^{-1}$.  Indeed, it follows from \S \ref{sec:1} that $g_\mathfrak{p}$ has nebentype 
 $[\cdot]_\mathfrak{p}=\theta^2_\mathfrak{p}$, and the equality
 $\theta^2_\mathfrak{p}=\chi_{0,\mathfrak{p}}^{-1}$ follows from the relations
$$
\chi_{0,\mathfrak{p}}(x)=\chi_{\mathfrak{p}}(x_p)=\Theta^2_\mathfrak{p}(\mathrm{art}_\Q(x_p))=\theta_\mathfrak{p}^{-2}(x)
$$
where $x\in\Z_p^\times$, $x_p\in\A^\times_\Q$ is the idele with $x$ in the $p$-component and $1$ in all other components, and we have used the fact that the composition 
$$
\Z_p^\times\map{x\mapsto x_p} \A_\Q^\times\map{\mathrm{art}_\Q} \mathrm{Gal}(\Q(\mu_{p^\infty})/\Q)\map{\epsilon_\mathrm{cyc}} \Z_p^\times
$$
is given by $x\mapsto x^{-1}$.

Although we will not recall here the construction of the class $\mathfrak{X}$, we will need an explicit description of  $\mathfrak{Z}_\mathfrak{p}$ when $\mathfrak{p}$ has weight two.  Fix such a $\mathfrak{p}$ and let $t$  be as above.  Let $J_t$ denote the Jacobian of the modular curve $X_t=X(\Phi_t)$.  Let $H_{p^t}$ denote the ring class field of the order $\co_{p^t}\subset \co_K$ of conductor $p^t$ and set $L_t=H_{p^t}(\mu_{p^t})$.   As in \cite[\S 2.2]{howard-hida} (but supressing $c=1$ from the notation) denote by $h_{t}\in X_t(L_{t})$ the point corresponding to the triple $(E_{t},\mathfrak{m}_t,\pi_{t})$ where $E_{t}(\C)=\C/\co_{p^t}$,  $\mathfrak{m}_t\subset E_t(\C)$ is the kernel of the cyclic $M$-isogeny
$$
\C/\co_{p^t}\rightarrow \C/(\co_{p^t}\cap \mathfrak{M})^{-1},
$$
and $\pi_t\in E_t(\C)$ is any generator of the kernel of the cyclic $p^t$-isogeny
$$
\C/\co_{p^t}\rightarrow \C/\co_K.
$$
Fix a rational cusp $\mathfrak{c}\in X_t(\Q)$ and define
$$
Q=\sum_{\sigma\in\mathrm{Gal}(L_{t}/K)} \chi_\mathfrak{p}^{-1}(\sigma) (h_{t}-\mathfrak{c})^\sigma   \in J_t(L_t)\otimes_\Z F_\mathfrak{p}.
$$
 By the Manin-Drinfeld theorem \cite[\S 2.1.5]{cremona} the point $Q$  does not depend on the choice of $\mathfrak{c}$.  The natural map $\mathbf{T}\rightarrow V_\mathfrak{p}$ can be factored as
$$
\mathbf{T}\rightarrow \mathrm{Ta}_p^\mathrm{ord}(J_t)\rightarrow V_\mathfrak{p}
$$
yielding a map
$$
J_t(L_t)\rightarrow  H^1(L_t,\mathrm{Ta}_p(J_t))  \map{e^\mathrm{ord}} H^1(L_t, \mathrm{Ta}_p^\mathrm{ord}(J_t)) 
\rightarrow H^1(L_t,V_\mathfrak{p}).
$$
If we define $Q_\mathfrak{p}$ to be the image of $Q$ under the induced map
\begin{equation}\label{kummer}
J_t(L_t)\otimes_\Z F_\mathfrak{p}\rightarrow  H^1(L_t,V_\mathfrak{p})
\end{equation}
then, tracing through the construction of $\mathfrak{X}$ found in \cite[\S 2.2]{howard-hida}, $Q_\mathfrak{p}$ is equal (up to multiplication by an element of $F_\mathfrak{p}^\times$) to the image of $\mathfrak{Z}_\mathfrak{p}$ under 
$$
H^1(K,V_\mathfrak{p}^\dagger)\map{\mathrm{res}}H^1(L_t,V_\mathfrak{p}^\dagger)
$$
after identifying the $G_{L_t}$-representations $V^\dagger_\mathfrak{p}\iso V_\mathfrak{p}\otimes\Theta_\mathfrak{p}^{-1}\iso V_\mathfrak{p}$.  As the kernel of the above restriction map  is both an $F_\mathfrak{p}$-vector space and is killed by $[L_t:K]$, it must be trivial, and we find that
\begin{equation}\label{simple class}
Q_\mathfrak{p}=0\iff \mathfrak{Z}_\mathfrak{p}=0.
\end{equation}

Let $\mathrm{Emb}(\mathfrak{p})$ denote the set of all $F$-algebra embeddings $F_\mathfrak{p}\hookrightarrow \Q_p^\alg$.  If for each $\iota\in \mathrm{Emb}(\mathfrak{p})$ we set $a^\iota_\mathfrak{p}(n)=\iota(a_\mathfrak{p}(n))$ then $\{a^\iota_\mathfrak{p}(n)\}$ generates  a finite extension of $\Q$ inside $\Q_p^\alg$, and so we may view 
$$
g^\iota_\mathfrak{p}=\sum_{n>0} a^\iota_\mathfrak{p}(n) q^n
$$ 
as a classical modular form with Fourier coefficients in $\Q^\alg$.  Denote by $\chi_\mathfrak{p}^\iota$ the unique $\Q^{\alg,\times}$-valued character of $\A_K^\times$ satisfying $\iota(\chi_\mathfrak{p}(x))=\chi^\iota_\mathfrak{p}(x)$ for all $x\in\A_K^\times$, and define $Q^\iota\in J_t(L_t)\otimes_\Z \Q^\alg$ in the same way as above, replacing $\chi_\mathfrak{p}$ by $\chi_\mathfrak{p}^\iota$.  Using the chosen embedding $\Q^\alg\hookrightarrow \C$, we may form the Rankin-Selberg convolution $L$-function
$L(s,\chi_\mathfrak{p}^\iota,g_\mathfrak{p}^\iota)$ as in \cite[\S 1]{howard-GZ}.  One could also accept (\ref{factorization}) below as the definition of $L(s,\chi_\mathfrak{p}^\iota,g_\mathfrak{p}^\iota)$.

\begin{Prop}\label{Prop:gz I}
Suppose $\mathfrak{p}\in\mathrm{Gen}_2(\theta)$ satisfies either (b) or (c) in Definition \ref{Def:generic}.  For any $\iota\in\mathrm{Emb}(\mathfrak{p})$ the $L$-function $L(s,\chi_\mathfrak{p}^\iota,g_\mathfrak{p}^\iota)$  vanishes to odd order at $s=1$ and
  $$
\mathfrak{Z}_\mathfrak{p}=0 \iff  L'(1,\chi_\mathfrak{p}^\iota,g_\mathfrak{p}^\iota)=0  .
 $$ 
\end{Prop}

\begin{proof}
Let $\mathbb{T}$ denote the $\Z$-algebra generated by the Hecke operators $T_n$ with $(n, Mp)=1$ and the diamond operators $\langle d\rangle$ with $(d,Mp)=1$ acting on the space of cusp forms $S_2(\Phi_t,\C)$.  Let $e_\mathfrak{p}\in \mathbb{T}\otimes_\Z F_\mathfrak{p}$ be the idempotent which projects onto the maximal summand on which $T_n=a_\mathfrak{p}(n)$ for all $(n,Mp)=1$ and $\langle d\rangle^{-1}=\chi_{0,\mathfrak{p}}(d)$ for all $(d,Mp)=1$.   Similarly, define $e_\mathfrak{p}^\iota\in \mathbb{T}\otimes_\Z\Q^\alg$ to be the projector to the maximal summand on which $T_n=a^\iota_\mathfrak{p}(n) $  and $\langle d\rangle^{-1}=\chi_{0,\mathfrak{p}}^\iota(d)$ for all $n$ and $d$ as above.    The algebra $\mathbb{T}$ acts on $J_t(L_t)$ and we have
\begin{equation}\label{iota switch}
e_\mathfrak{p}Q=0\iff e_\mathfrak{p}^\iota Q^\iota=0.
\end{equation}

Our hypotheses imply that  $g_\mathfrak{p}$ is new of level $Mp^t$ (where $t=t(\mathfrak{p})$ as above), so by Eichler-Shimura theory the summand
\begin{equation*}
e_\mathfrak{p} (\mathrm{Ta}^\mathrm{ord}_p(J_t)\otimes_{\co_F}F_\mathfrak{p})
=
e_\mathfrak{p} (\mathrm{Ta}_p(J_t)\otimes_{\Z_p}F_\mathfrak{p})
\end{equation*}   
is taken isomorphically to $V_\mathfrak{p}$ under  $\mathrm{Ta}_p^\mathrm{ord}(J_t)\otimes_{\co_F}F_\mathfrak{p}\rightarrow V_\mathfrak{p}$.  It follows that the map
$$
e_\mathfrak{p}( J_t(L_t)\otimes_\Z F_\mathfrak{p} )\rightarrow  H^1(L_t, V_\mathfrak{p})
$$
induced by (\ref{kummer}) is injective.  As this map takes $e_\mathfrak{p} Q$ to $Q_\mathfrak{p}$,  (\ref{simple class}) and (\ref{iota switch}) imply that
$$
 e_\mathfrak{p}^\iota Q^\iota=0\iff \mathfrak{Z}_{\mathfrak{p}}=0.
$$
We may now apply the results of  \cite[\S 1]{howard-GZ} with $f=g^\iota_\mathfrak{p}$, $\chi=\chi_\mathfrak{p}^\iota$, $N=Mp^t$, and $C=S=p^t$.   As we have assumed that $\epsilon_K(M)=1$ the functional equation  \cite[(1)]{howard-GZ}  forces  $L(s,\chi_\mathfrak{p}^\iota,g_\mathfrak{p}^\iota)$ to vanish to odd order at $s=1$, and  \cite[Theorem A]{howard-GZ} asserts that $e_\mathfrak{p}^\iota Q_{\mathfrak{p}}=0$ if and only if $L'(1,\chi^\iota,g_\mathfrak{p}^\iota)=0$.   
\end{proof}

\begin{Prop}\label{Prop:gz II}
Suppose $\mathfrak{p}\in\mathrm{Gen}_2(\theta)$ satisfies condition (a) in Definition \ref{Def:generic}. For any $\iota\in\mathrm{Emb}(\mathfrak{p})$ the $L$-function $L(s,\chi^\iota_\mathfrak{p},g^\iota_\mathfrak{p})$ vanishes to odd order at $s=1$ and 
$$
\mathfrak{Z}_\mathfrak{p}=0  \implies L'(1,\chi^\iota_\mathfrak{p},g^\iota_\mathfrak{p})=0 .
$$
\end{Prop}

\begin{proof}
Our hypotheses imply that $t=1$, $\chi_\mathfrak{p}$ is trivial, and  there is a normalized newform 
$$
f=\sum_{n>0}  b(n) q^n \in S_2(\Gamma_0(M),F_\mathfrak{p})
$$ 
such that $g_\mathfrak{p}$ is the $p$-stabilization of $f$.  Define $f^\iota\in S_2(\Gamma_0(M),\Q^\alg)$ to be the modular form with Fourier coefficients  $b^\iota(n)=\iota(b(n))$.  The $L$-functions of $g^\iota_\mathfrak{p}$ and $f^\iota$ are related by
$$
\left( 1-\frac{p^{1-s}}{a^\iota_\mathfrak{p}(p) } \right) L(s,f^\iota)=L(s,g^\iota_\mathfrak{p})
$$
while the Rankin-Selberg $L$-functions are related by
$$
\left( 1-\frac{p^{1-s}}{a^\iota_\mathfrak{p}(p) } \right) \left( 1-\epsilon_K(p) \frac{p^{1-s}}{a^\iota_\mathfrak{p}(p) } \right) 
L(s,\chi_\mathfrak{p}^\iota, f^\iota)=L(s,\chi_\mathfrak{p}^\iota, g^\iota_\mathfrak{p}).
$$
As $a^\iota_\mathfrak{p}(p)$ is a root of $X^2-b^\iota(p)X+p$, and hence has complex absolute value $p^{1/2}$,  we have $a_\mathfrak{p}(p)\not=\pm 1$.  Therefore the Euler factors do not vanish at $s=1$ and the two Rankin-Selberg $L$-functions have the same order of vanishing, which must be odd as the sign of the functional equation of the $L$-function on the left is $-\epsilon_K(M)=-1$ by \cite[IV (0.2)]{gz} or \cite[(1)]{howard-GZ}.  The central derivative of the $L$-function on the left is determined by the original Gross-Zagier theorem, and we will relate our Heegner class $\mathfrak{Z}_\mathfrak{p}$ to the classical Heegner point considered by Gross-Zagier.

Consider the degeneracy maps
$$
\alpha,\beta: X_0(Mp)\rightarrow X_0(M)
$$
defined on moduli by 
$$
\alpha(E,C\times D)=(E,C)\hspace{1cm}
\beta(E,C\times D)=(E/D, (C\times D)/D)
$$
where $E$ is an elliptic curve (over some $\Q$-scheme) and $C$ and $D$ are cyclic subgroup schemes of $E$ of orders $M$ and $p$, respectively.  The Hecke algebra  $\mathbb{T}$ defined in the proof of Proposition \ref{Prop:gz I} acts on $J_t$, $J_0(Mp)$, and $J_0(M)$, and $\alpha$ and $\beta$ induce $\mathbb{T}$-equivariant maps
$$
\alpha^*,\beta^*: \mathrm{Ta}_p(J_0(M))\rightarrow  \mathrm{Ta}_p(J_0(Mp))
$$
which satisfy the relations
\begin{equation}\label{pullback relations}
U_p\circ\alpha^*=p\beta^*\hspace{1cm}
U_p\circ\beta^*=\beta^*\circ T_p-\alpha^*.
\end{equation}
Gently abusing notation, we abbreviate 
$$
e_\mathfrak{p}\mathrm{Ta}_p(J)=e_\mathfrak{p}(\mathrm{Ta}_p(J)\otimes_{\Z_p}F_\mathfrak{p})
$$
for $J$ any one of $J_t$, $J_0(M)$, or $J_0(Mp)$.  Whenever  $(n,Mp)=1$ we have $b(n)=a_\mathfrak{p}(n)$, thus $T_n$ acts as $b(n)$ on $e_\mathfrak{p} \mathrm{Ta}_p(J_0(M))$.  It follows from strong multiplicity one that also $T_p$ acts as $b(p)$ on $e_\mathfrak{p} \mathrm{Ta}_p(J_0(M))$.
From this and (\ref{pullback relations}) it follows that $U_p$ acts  semi-simply on
\begin{eqnarray}\nonumber
e_\mathfrak{p} \mathrm{Ta}_p(J_t) 
&\iso&
e_\mathfrak{p} \mathrm{Ta}_p(J_0(Mp))  \\ 
&\iso&
\alpha^*e_\mathfrak{p} \mathrm{Ta}_p(J_0(M)) 
\oplus
\beta^*e_\mathfrak{p} \mathrm{Ta}_p(J_0(M)) \label{pull-back decomp}
\end{eqnarray}
with characteristic polynomial 
$$
(X^2-b(p)X+p)^2=(X-a_\mathfrak{p}(p))^2 (X-p/a_\mathfrak{p}(p))^2,
$$
and that $e^\mathrm{ord}$ and $1-e^\mathrm{ord}$ act as projections onto the $a_\mathfrak{p}(a)$ and $p/a_\mathfrak{p}(p)$  eigenspaces, respectively.  Thus the map $\mathrm{Ta}_p^\mathrm{ord}(J_t)\otimes_{\co_F}F_\mathfrak{p}\rightarrow V_\mathfrak{p}$ takes the summand
\begin{equation*}
e_\mathfrak{p} (\mathrm{Ta}^\mathrm{ord}_p(J_t)\otimes_{\co_F}F_\mathfrak{p}) = e^\mathrm{ord} e_\mathfrak{p}\mathrm{Ta}_p(J_t)
\end{equation*}   
isomorphically to $V_\mathfrak{p}$, and it follows that the map
$$
e^\mathrm{ord} e_\mathfrak{p}( J_t(L_t)\otimes_\Z F_\mathfrak{p} )\rightarrow  H^1(L_t, V_\mathfrak{p})
$$
induced by (\ref{kummer}) is injective.  As this map takes $e^\mathrm{ord} e_\mathfrak{p} Q$ to $Q_\mathfrak{p}$,   we deduce from (\ref{simple class}) 
$$
 \mathfrak{Z}_\mathfrak{p} =0 \iff  e^\mathrm{ord} e_\mathfrak{p} Q=0 \iff \left( U_p-\frac{p}{a_\mathfrak{p}(p)} \right)  e_\mathfrak{p} Q=0.
$$

Let $x\in X_0(Mp)(H_p)$ be the image of $h_t$ (still $t=1$) under the degeneracy map $X_t\rightarrow X_0(Mp)$.  Thus $x$ corresponds to the cyclic $Mp$-isogeny given by the composition
$$
\C/\co_p\rightarrow \C/\co_K\rightarrow \C/\mathfrak{M}^{-1}
$$
Fix a rational cusp $\mathfrak{c}\in X_0(Mp)(\Q)$ and define
$$
P=\sum_{\sigma\in\mathrm{Gal}(H_p/K)} (x-\mathfrak{c})^\sigma\in J_0(Mp)\otimes_\Z F_\mathfrak{p}.
$$
As the first isomorphism of (\ref{pull-back decomp}) identifies the Kummer images of $[L_t:H_p] e_\mathfrak{p}P$ and $e_\mathfrak{p}Q$, we have
$$
\mathfrak{Z}_\mathfrak{p} =0 \iff \left( U_p-\frac{p}{a_\mathfrak{p}(p)} \right)  e_\mathfrak{p} P=0.
$$
The Albanese maps $\alpha_*,\beta_*:J_0(Mp)\rightarrow J_0(M)$  are related by $\beta_* U_p=p\alpha_*$, and applying $\beta_*$ to the right hand side of the last $\iff$ we find
\begin{equation}\label{almost there}
\mathfrak{Z}_\mathfrak{p} =0 \implies \left( \alpha_*-\frac{\beta_*}{a_\mathfrak{p}(p)} \right)  e_\mathfrak{p} P=0.
\end{equation}
The point $\beta(x)\in X_0(M)(H)$ is none other than the classical Heegner point considered by Gross-Zagier corresponding to the cyclic $M$-isogeny $\C/\co_K\rightarrow \C/\mathfrak{M}^{-1}$,  while $\alpha(x)\in X_0(M)(H_p)$ is the point corresponding to $\C/\co_p\rightarrow (\co_p\cap \mathfrak{M})^{-1}$.  The Euler system relations of  \cite[\S 3.1]{perrin-riou} imply that
$$
u\cdot \mathrm{Tr}_{H_p/H}\alpha(x)= 
\left\{\begin{array}{ll}
T_p \cdot \beta(x)  &  \mathrm{if\ }\epsilon_K(p)=-1 \\
(T_p -\sigma_p-\sigma_p^*  )  \cdot \beta(x) & \mathrm{if\ }\epsilon_K(p)=1
\end{array}\right.
$$
as divisors on $X_0(M)_{/H}$,  where $2u=|\co_K^\times|$ and $\sigma_p,\sigma^*_p\in \mathrm{Gal}(H/K)$ are the Frobenius automorphisms of the two primes above $p$ when $\epsilon_K(p)=1$.  From this and $u[H_p:H]=p-\epsilon_K(p)$ we deduce
$$
 \alpha_*(P)=\frac{T_p-1-\epsilon_K(p)}{p-\epsilon_K(p)} \cdot \beta_*(P).
$$
We may use this,  $a_\mathfrak{p}(p)^2-b(p)a_\mathfrak{p}(p)+p=0$, and $a_\mathfrak{p}(p)\not=\pm 1$ to deduce from  (\ref{almost there}) the implications
\begin{eqnarray*}
\mathfrak{Z}_\mathfrak{p}=0 &\implies& (a_\mathfrak{p}(p)-1)(a_\mathfrak{p}(p)-\epsilon_K(p)) \cdot \beta_*(e_\mathfrak{p}P) =0 \\
&\implies& \beta_*(e_\mathfrak{p}P) =0.
\end{eqnarray*}
Finally, according to the Gross-Zagier theorem \cite{gz} the element  
$$
e_\mathfrak{p}\beta_*(P)\in J_0(M)\otimes_\Z F_\mathfrak{p}
$$
is trivial if and only if $L'(1,\chi^\iota_\mathfrak{p},f^\iota)=0$.  This completes the proof of the proposition.
\end{proof}

 It seems likely that the implication $L'(1,\chi_\mathfrak{p}^\iota,g_\mathfrak{p}^\iota)=0\implies \mathfrak{Z}_\mathfrak{p}=0$  in Proposition \ref{Prop:gz II} holds as well, although we are unable to provide a proof.  There are two situations in which a weight $2$ arithmetic prime fails to be generic for $\theta$, and in neither case does it seem that one should expect $\mathfrak{Z}_\mathfrak{p}$ to be related to the derivative  $L'(1,\chi^\iota_\mathfrak{p},g_\mathfrak{p}^\iota)$. The first case is when  $g_\mathfrak{p}$ is newform of level $Mp$ and trivial nebentype and $\theta_\mathfrak{p}$ is trivial.  Letting  $\mathbf{1}_K$ denote the trivial character of $\A_K^\times$, the sign in the functional equation of 
 $$
 L(s, \chi_\mathfrak{p}^\iota,g_\mathfrak{p}^\iota)  = L(s,\mathbf{1}_K,g_\mathfrak{p}^\iota)
 $$ 
 is  $-\epsilon_K(p)$ and there are two sub-cases to consider.  If $p$ is inert in $K$ then $L(s, \mathbf{1}_K,g_\mathfrak{p}^\iota)$ vanishes to even order at $s=1$.  If  $p$ splits in $K$ then  the  Gross-Zagier theorem \cite{gz}   relates $L'(1,\mathbf{1}_K, g^\iota_\mathfrak{p})$ to the  point on $X_0(Mp)$ corresponding to the isogeny 
$$
\C/\co_K\rightarrow \C/(\mathfrak{MP})^{-1}
$$
where $\mathfrak{P}$ is a prime of $K$ above $p$.  In particular, both the source and target of the isogeny have complex multiplication by the maximal order $\co_K$.  On the other hand our cohomology class $\mathfrak{Z}_{\mathfrak{p}}$ is constructed from a point on $X(\Gamma_0(M)\cap\Gamma_1(p))$ whose image in $X_0(Mp)$ corresponds to the isogeny 
$$
\C/\co_p\rightarrow \C/\mathfrak{M}^{-1}
$$
in which the source and target have different CM orders.  Thus in this sub-case our Heegner class $\mathfrak{Z}_{\mathfrak{p}}$ is the wrong object to look at.
The other situation in which $\mathfrak{p}$ is not generic for $\theta$ is when  $\theta_\mathfrak{p}=\omega^{\frac{p-1}{2}}$ and $g$ is the $p$-stabilization of a newform $f\in S_2(\Gamma_0(M),F_\mathfrak{p})$.  We then have
$$
 L(s, \chi_\mathfrak{p}^\iota,g_\mathfrak{p}^\iota)
= L(s, \chi_\mathfrak{p}^\iota, f^\iota)
= L(s,\mathbf{1}_K ,f^\iota\otimes\omega^{\frac{p-1}{2}}).
$$
 The sign in the functional equation of $L(s, \mathbf{1}_K, f^\iota \otimes\omega^{\frac{p-1}{2}})$ is again  $-\epsilon_K(p)$, and exactly as in the case considered above one does not expect the central derivative of this $L$-function to be related to the class $\mathfrak{Z}_\mathfrak{p}$.

Suppose $\mathfrak{p}\subset R$ is arithmetic of weight two.  If 
\begin{equation}\label{order one silly}
\mathrm{ord}_{s=1}L(s,\chi^\iota_\mathfrak{p}, g^\iota_\mathfrak{p})=1
\end{equation}  
for every  $\iota\in\mathrm{Emb}(\mathfrak{p})$ then  we will say that $(g_\mathfrak{p},\chi_\mathfrak{p})$ has \emph{analytic rank one}.   If we assume that $\mathfrak{p}\in\mathrm{Gen}_2(\theta)$ then it suffices to verify (\ref{order one silly})  for a single $\iota\in\mathrm{Emb}(\mathfrak{p})$.  Indeed if $\mathfrak{p}$ satisfies condition (b) or (c) in Definition \ref{Def:generic} then this is clear from Proposition \ref{Prop:gz I}.  If $\mathfrak{p}$ satisfies condition (a) in Definition \ref{Def:generic} then, in the notation of the proof of Proposition \ref{Prop:gz II}, 
$$
\mathrm{ord}_{s=1} L(s,\chi_\mathfrak{p}^\iota,g_\mathfrak{p}^\iota)
=\mathrm{ord}_{s=1} L(s,\chi_\mathfrak{p}^\iota,f^\iota)
$$
where $f\in S_2(\Gamma_0(M), F_\mathfrak{p})$ and $\chi_\mathfrak{p}$ is the trivial character.  It follows from the Gross-Zagier theorem \cite{gz} that if the order of vanishing of the $L$-function on the right is $1$ for some $\iota$, then the order of vanishing is $1$ for all $\iota$.

\begin{Cor}\label{generic nonzero}
The following are equivalent:
\begin{enumerate}
\item 
there is a $\mathfrak{p}\in\mathrm{Gen}_2(\theta)$ such that $(g_\mathfrak{p},\chi_\mathfrak{p})$ has analytic rank one,
\item
there is an arithmetic prime $\mathfrak{p}$ such that $\mathfrak{Z}_\mathfrak{p}\not=0$,
\item
$\mathfrak{Z}$ is not $R$-torsion,
\item
$\mathfrak{Z}_\mathfrak{p}\not=0$ for all but finitely many arithmetic primes $\mathfrak{p}$,
\item 
$(g_\mathfrak{p},\chi_\mathfrak{p})$ has analytic rank one for all but finitely many $\mathfrak{p}\in\mathrm{Gen}_2(\theta)$.
\end{enumerate}
\end{Cor}

\begin{proof}
The implication (a)$\implies$(b) is immediate from Propositions \ref{Prop:gz I} and \ref{Prop:gz II}.   Assume (b) holds.  It follows from \cite[Proposition 2.4.5]{howard-hida} that $\mathfrak{Z}\in\mathrm{Sel}_\mathrm{Gr}(K,\mathbf{T}^\dagger)$, and (b) implies that  $\mathfrak{Z}$ has nontrivial image in the localization $\mathrm{Sel}_\mathrm{Gr}(K,\mathbf{T}^\dagger)_\mathfrak{p}$ for some arithmetic prime $\mathfrak{p}$.  According to \cite[Proposition 12.7.13.4(iii)]{nek} (and using \cite[(21)]{howard-hida} to compare the strict Greenberg Selmer group with Nekov\'a\v{r}'s extended Selmer group) the localization $\mathrm{Sel}_\mathrm{Gr}(K,\mathbf{T}^\dagger)_\mathfrak{p}$ is free of finite rank over $R_\mathfrak{p}$.  Thus $\mathfrak{Z}$ has non-torsion image after localizing at $\mathfrak{p}$, and so was non-torsion to begin with.  Thus  (b)$\implies$(c).   
Let $\mathcal{H}$ denote the maximal extension of $H$ unramified outside $Mp$ and set $\mathfrak{G}=\mathrm{Gal}(\mathcal{H}/H)$.  For any arithmetic prime $\mathfrak{p}$ we may pick a uniformizing parameter $\varpi_\mathfrak{p}$ of the discrete valuation ring (by \cite[\S 12.7.5]{nek})  $R_\mathfrak{p}$.  The $\mathfrak{G}$-cohomology of the exact sequence
$$
0\rightarrow \mathbf{T}_{\mathfrak{p}}^\dagger\map{\varpi_\mathfrak{p}}\mathbf{T}_{\mathfrak{p}}^\dagger\rightarrow V_\mathfrak{p}^\dagger\rightarrow 0
$$
shows that the natural map
$$
H^1(\mathfrak{G},\mathbf{T}^\dagger)_\mathfrak{p}/ \mathfrak{p} H^1(\mathfrak{G},\mathbf{T}^\dagger)_\mathfrak{p}\rightarrow H^1(\mathfrak{G},V^\dagger_\mathfrak{p})
$$
is injective.  As $H^1(\mathfrak{G},\mathbf{T}^\dagger)$ is finitely generated as an $R$-module and contains $\mathfrak{Z}$, the implication (c)$\implies$(d) follows from \cite[Lemma 2.1.7]{howard-hida}.  The implication (d)$\implies$(e) follows by another application of Proposition \ref{Prop:gz I}, and (e)$\implies$(a) is obvious.
\end{proof}

%%%%%%%%%%%%%%%%%%%%%%%%%%%%%%%%%%%%%%%%%%%%%

\section{Results over $\Q$}
\label{sec:3}

%%%%%%%%%%%%%%%%%%%%%%%%%%%%%%%%%%%%%%%%%%%%%

For any arithmetic prime $\mathfrak{p}$ let $\co_\mathfrak{p}$ denote ring of integers of $F_\mathfrak{p}$.  The Mazur-Tate-Teitelbaum  \cite{mtt} $p$-adic $L$-function 
$
\mathcal{L}(\cdot, g_\mathfrak{p})\in \co_{\mathfrak{p}}[[\Z_p^\times]]
$ 
of $g_\mathfrak{p}$, viewed as an $F_\mathfrak{p}$-valued function on characters  $\chi:\Z_p^\times\rightarrow F_\mathfrak{p}^\times$, satisfies the functional equation \cite[Proposition 2.3.6]{howard-hida}
$$
\mathcal{L}(\chi,g_\mathfrak{p})=-w\chi^{-1}(-M) \theta_\mathfrak{p}(-M) \cdot \mathcal{L}( [\cdot]_\mathfrak{p} \chi^{-1} , g_\mathfrak{p})
$$
where $w=\pm 1$ is independent of $\mathfrak{p}$ (but depends on $\theta$; see \cite[Remark 2.3.4]{howard-hida}).  Taking $\chi=\theta_\mathfrak{p}$ gives 
\begin{equation}\label{center point}
\mathcal{L}(\theta_\mathfrak{p},g_\mathfrak{p})=-w\mathcal{L}(\theta_\mathfrak{p},g_\mathfrak{p}).
\end{equation}

For each arithmetic prime $\mathfrak{p}\subset R$ of even weight $r$ the character $\theta_\mathfrak{p}$ has the form 
 $$
 \theta_\mathfrak{p}(z)=z^{\frac{r-2}{2}} \vartheta_\mathfrak{p}(z)
 $$
 with $\vartheta_\mathfrak{p}$ of finite order. Define a $p$-adic eigenform of trivial nebentype
$$
f_\mathfrak{p}=g_\mathfrak{p}\otimes\vartheta^{-1}_\mathfrak{p}\in S_r(\Gamma_0(Mp^{2t}),F_\mathfrak{p})
$$
where $t=t(\mathfrak{p})$ as in \S \ref{sec:1}.  For any $\iota\in \mathrm{Emb}(\mathfrak{p})$ define 
$$
f_\mathfrak{p}^\iota\in S_r(\Gamma_0(Mp^{2t}), \Q^\alg)
$$ 
to be the classical eigenform obtained by applying $\iota$ to the Fourier coefficients of $f_\mathfrak{p}$ (all of which are algebraic over $\Q$).  Let us say that  $\mathfrak{p}$  \emph{has a trivial zero} if $r=2$, $\theta_\mathfrak{p}$ is trivial, and $a_\mathfrak{p}(p)=1$. We write $\Omega(\theta)$ for the set of arithmetic primes $\mathfrak{p}\subset R$ of even weight  which do \emph{not} have a trivial zero, and note that  $\mathrm{Gen}_2(\theta)\subset  \Omega(\theta).$

\begin{Lem}
If $\mathfrak{p}\in\Omega(\theta)$ has weight $r$ and $\iota\in\mathrm{Emb}(\mathfrak{p})$ then
\begin{equation}\label{interpolation}
\mathcal{L}(\theta_\mathfrak{p},g_\mathfrak{p})  = 0
\iff   L\left(r/2,f_\mathfrak{p}^\iota\right)=0.
\end{equation}
Furthermore the functional equation of $L(s,f_\mathfrak{p}^\iota)$ has sign $-w$.
\end{Lem}

\begin{proof}
The equivalence (\ref{interpolation}) follows from the interpolation formula of Mazur-Tate-Teitelbaum \cite[\S I.14]{mtt}.  The assumption that $\mathfrak{p}$ does not have a trivial zero ensures that the $p$-adic multiplier appearing in the interpolation formula is nonzero, so that there is no extra zero in the sense of \cite[\S I.15]{mtt}. According to \cite[Proposition 12.7.14.4(i)]{nek} the sign in the functional equation of $L(s,f_\mathfrak{p}^\iota)$ is independent of $\mathfrak{p}\in\Omega(\theta)$, and according to \cite[\S I.18]{mtt} the sign is equal to $-w$ for any $\mathfrak{p}\in\Omega(\theta)$ such that $g_\mathfrak{p}$ has trivial  nebentype.
\end{proof}

Given  $\mathfrak{p}\in\Omega(\theta)$ of weight $r$ we say that $f_\mathfrak{p}$ has \emph{analytic rank zero} if $L(r/2,f_\mathfrak{p}^\iota)\not=0$.  It follows from (\ref{interpolation}) that this condition is independent of the choice of $\iota$.   The following result is an easy consequence of work of Kato,  Kitagawa, and Mazur; we state it for the purpose of comparison with Theorem \ref{central derivatives} below.

\begin{Thm}\label{rank zero}
The following are equivalent:
\begin{enumerate}
\item
there is a $\mathfrak{p}\in\Omega(\theta)$ such that $f_\mathfrak{p}$ analytic rank zero,
\item
$f_\mathfrak{p}$ has analytic rank zero for all but finitely many $\mathfrak{p}\in\Omega(\theta)$.
\end{enumerate}
Furthermore, when these conditions hold $\mathrm{Sel}_\mathrm{Gr}(\Q,\mathbf{T}^\dagger)$ is a rank zero $R$-module.
\end{Thm}

\begin{proof}
We give a quick sketch of the proof.  In \cite{epw} one finds the construction of a two-variable $p$-adic $L$-function (originally constructed by  Kitagawa \cite{kitagawa} following unpublished work of Mazur) $\mathcal{L}\in R[[\Z^\times_p]]$ whose image $\mathcal{L}(\theta,\cdot)$  under the map $R[[\Z_p^\times]]\rightarrow R$ induced by $\theta$ has the property that $R\rightarrow F_\mathfrak{p}$ takes $\mathcal{L}(\theta,\cdot )$ to $\mathcal{L}(\theta_\mathfrak{p},g_\mathfrak{p})$ for any arithmetic prime $\mathfrak{p}$.  If   $\mathcal{L}(\theta_\mathfrak{p},g_\mathfrak{p})\not=0$ for some arithmetic prime $\mathfrak{p}$ then $\mathcal{L}(\theta,\cdot )\not=0$, and hence by \cite[Lemma 2.1.7]{howard-hida} $\mathcal{L}(\theta_\mathfrak{p},g_\mathfrak{p})\not=0$ for \emph{all but finitely many} arithmetic primes $\mathfrak{p}$.  Thus the equivalence of (a) and (b) follows from  (\ref{interpolation}).  

Choosing a $\mathfrak{p}\in\Omega(\theta)$ such that $\mathcal{L}(\theta_\mathfrak{p},g_\mathfrak{p})\not=0$, the work of Kato \cite{kato} shows that the Bloch-Kato Selmer group $H^1_f(\Q,V_\mathfrak{p}^\dagger)$ is trivial  (alternatively, one could choose $\mathfrak{p}$ to have weight $2$ and then deduce the triviality of the Bloch-Kato Selmer group using the results of Gross-Zagier \cite{gz} and the methods of Kolyvagin, as  extended by  Nekov\'a\v{r} \cite{nek-euler} to include modular forms with coefficients in number fields).  It follows from \cite[Proposition 12.7.13.4(1)]{nek} that the natural map
$$
\mathrm{Sel}_\mathrm{Gr}(\Q,\mathbf{T}^\dagger)\otimes_R F_\mathfrak{p}\rightarrow  H^1_f(\Q,V_\mathfrak{p}^\dagger)
$$
is injective, and therefore the rank of $\mathrm{Sel}_\mathrm{Gr}(\Q,\mathbf{T}^\dagger)$ is zero.
\end{proof}

If $w=-1$ then it is a conjecture of Greenberg \cite{greenberg_deformation} that  condition (b) always holds in Theorem \ref{rank zero}, and condition (a) gives an effective method of verifying this conjecture for a given Hida family.  Of course if $w=1$ then $\mathcal{L}(\theta_\mathfrak{p},g_\mathfrak{p})$ vanishes for every arithmetic prime $\mathfrak{p}$ by (\ref{center point}) and Theorem \ref{rank zero} is vacuous; one should instead look at the behavior of $L'(r/2,f_\mathfrak{p}^\iota)$ as $\mathfrak{p}\in\Omega(\theta)$ varies.  As it is not known that the order of vanishing of $L(s,f_\mathfrak{p}^\iota)$ at $s=r/2$ is equal to the order of vanishing of $\mathcal{L}(\chi, g_\mathfrak{p})$ at $\chi=\theta_\mathfrak{p}$, studying the two-variable $p$-adic $L$-function $\mathcal{L}$ yields little information about $L'(r/2,f_\mathfrak{p})$.  We instead use the big Heegner point as a substitute for the element $\mathcal{L}(\theta,\cdot)$ in order to obtain information about the generic behavior of the central derivatives, and substitute the Gross-Zagier type results of \S \ref{sec:2} for the interpolation formula (\ref{interpolation}). As noted earlier, the cost is that we lose all information about forms of weight greater than two.

We now assume that $w=1$.  Fix a $\mathfrak{p}\in \mathrm{Gen}_2(\theta)$ and $\iota\in\mathrm{Emb}(\mathfrak{p})$.  The sign in the functional equation of $L(s,f^\iota_\mathfrak{p})$ is equal to $-1$, and we further assume that $L'(1, f^\iota_\mathfrak{p})\not=0$. By a result of Bump-Friedberg-Hoffstein \cite{bfh} one may choose a quadratic imaginary field $K$ of discriminant prime to $Mp$ in which all prime divisors of $M$ are split, and such that 
$L(1,f_\mathfrak{p}^\iota\otimes\epsilon_K)\not=0.$  
We now apply the results of \S \ref{sec:2}.  Using the factorization of $L$-functions
\begin{equation}\label{factorization}
L(s,\chi_\mathfrak{p}^\iota, g_\mathfrak{p}^\iota) = L(s,f_\mathfrak{p}^\iota) L(s,f_\mathfrak{p}^\iota\otimes\epsilon_K)
\end{equation}
we find that $(g_\mathfrak{p},\chi_\mathfrak{p})$ has analytic rank one, and it follows that $L'(1,f_\mathfrak{p}^\iota)\not=0$ for \emph{every} choice of $\iota$. In this situation we say that $f_\mathfrak{p}$ has \emph{analytic rank one}.
We now come to our main result.

\begin{Thm}\label{central derivatives}
Suppose that $w=1$.  The following are equivalent:
\begin{enumerate}
\item
 there exists a $\mathfrak{p}\in\mathrm{Gen}_2(\theta)$ such that $f_\mathfrak{p}$ has analytic rank one;
 \item
 $f_\mathfrak{p}$ has analytic rank one for all but finitely many $\mathfrak{p}\in\mathrm{Gen}_2(\theta)$.
 \end{enumerate}
 Furthermore, when these  conditions are satisfied  $\mathrm{Sel}_\mathrm{Gr}(\Q,\mathbf{T}^\dagger)$ is a rank one $R$-module.
\end{Thm}

\begin{proof}
Suppose $\mathfrak{p}\in\mathrm{Gen}_2(\theta)$ is such that $f_\mathfrak{p}$ has analytic rank one.  The argument above shows that we may choose the quadratic imaginary field $K$ of \S \ref{sec:2} in such a way that $\mathfrak{Z}_{\mathfrak{p}}\not=0$.  By Corollary \ref{generic nonzero} we see that $(g_\mathfrak{p},\chi_\mathfrak{p})$ has analytic rank one for all but finitely many $\mathfrak{p}\in \mathrm{Gen}_2(\theta)$.  The factorization (\ref{factorization}), together with the odd order of vanishing of $L(s,f_\mathfrak{p}^\iota)$ shows that $f_\mathfrak{p}$ has analytic rank one whenever $(g_\mathfrak{p},\chi_\mathfrak{p})$ has analytic rank one, and hence (a)$\implies$(b).  The other implication is obvious, and the final claim follows from Corollary \ref{generic nonzero}, \cite[Corollary 3.4.3]{howard-hida}, and \cite[(21)]{howard-hida}.
\end{proof}

Once again it is a conjecture of Greenberg \cite{greenberg_deformation} that condition (b) holds in Theorem \ref{central derivatives}, and condition (a) gives an effective method for verifying this for a given Hida family.  We point out that the implication
$$
(a)\implies \mathrm{rank}_R \ \mathrm{Sel}_\mathrm{Gr}(\Q,\mathbf{T}^\dagger)=1
$$
 in Theorem \ref{central derivatives} can also be deduced using Nekov\'a\v{r}'s parity results \cite[\S 12]{nek} (although this implication is never stated explicitly there), and so it is only the implication (a)$\implies$(b) which is new.

\bibliographystyle{plain}

\end{document}